\documentclass[11pt]{amsart}
\usepackage[margin=1.05in]{geometry}
\usepackage{amsmath,amssymb,amsthm}
\usepackage{hyperref}
\usepackage{microtype}

\hypersetup{hidelinks}

\newtheorem{theorem}{Theorem}[section]
\newtheorem{lemma}[theorem]{Lemma}
\newtheorem{corollary}[theorem]{Corollary}

\newcommand{\N}{\mathbb{N}}
\newcommand{\Z}{\mathbb{Z}}
\newcommand{\F}{\mathbb{F}}

\DeclareMathOperator{\Span}{span}

\title{The Sidon Decomposition Problem in Abelian Groups of Bounded Torsion}
\author{Mark Lewko}
\address{Lebanon, New Hampshire USA}
\email{mlewko@gmail.com}

\date{}

\begin{document}

\begin{abstract}
Let \(G\) be a compact abelian group whose dual group \(\Gamma=\widehat{G}\) has bounded torsion. In 1967, Malliavin-Brameret and Malliavin proved that every Sidon set in \(\Gamma\) is a finite union of quasi-independent sets when \(\Gamma\) has prime exponent. This was later extended to squarefree exponents in work of Varopoulos and Bourgain. We prove the remaining bounded-torsion case. Consequently, if \(\widehat{G}\) has bounded torsion, then a subset \(\Lambda\subset \widehat{G}\setminus\{0\}\) is Sidon if and only if it is a finite union of quasi-independent sets.
\end{abstract}

\maketitle

\section{Introduction}\label{sec:intro}
Let \(G\) be a compact abelian group and let \(\Gamma=\widehat G\) be its discrete dual group. We call a subset \(\Lambda\subseteq \Gamma\setminus\{0\}\) Sidon\footnote{Throughout, we adopt the convention that Sidon sets do not contain \(0\). This avoids degeneracies involving the identity element that can complicate otherwise clean statements.} if there is a constant \(S_\Lambda<\infty\) such that, for every finite set \(F\subseteq\Lambda\) and every choice of scalar coefficients \((a_\gamma)_{\gamma\in F}\),
$$
        \sum_{\gamma\in F}|a_\gamma|
        \leq
        S_\Lambda\left\|\sum_{\gamma\in F}a_\gamma\gamma\right\|_{L^\infty(G)}.
$$
The reverse inequality is immediate from the triangle inequality. Thus the Sidon condition says that a trigonometric polynomial supported on \(\Lambda\) has \(L^\infty\)-norm comparable to the $\ell^1$-norm of its coefficients.

Sidon sets have been extensively studied; we refer the reader to \cite{GH,LR,Pisier2016Subgaussian} for background and will not attempt a survey here. A set \(Q\subseteq\Gamma\setminus\{0\}\) is called \emph{quasi-independent} if the only solution to
$$
        \sum_{\gamma\in Q}\epsilon_\gamma\gamma=0,
        \qquad \epsilon_\gamma\in\{-1,0,1\},
$$
with finitely many nonzero coefficients is the trivial one; equivalently, no nonempty signed sum of distinct elements of \(Q\) vanishes. Quasi-independent sets are the most elementary examples of Sidon sets. This follows from the standard Riesz product argument; see, for example, \cite{GH,LR}. A celebrated theorem of Drury \cite{Drury} states that the finite union of Sidon sets is Sidon. Putting these facts together, every finite union of quasi-independent sets is Sidon. Whether the converse holds---that is, whether every Sidon set is a finite union of quasi-independent sets---is a longstanding open problem. In the early 1980s, Pisier proved the following result \cite{Pisier1981Sidon,Pisier1983Arithmetic,Pisier1983Entropy}.

\begin{theorem}[Pisier's arithmetic characterization]\label{thm:pisier}
Let \(G\) be a compact abelian group, let \(\Gamma=\widehat G\), and let
\(\Lambda\subseteq \Gamma\setminus\{0\}\). Then \(\Lambda\) is Sidon if and only if there is a constant \(\delta>0\) such that every finite subset \(A\subseteq\Lambda\) contains a quasi-independent subset \(Q\subseteq A\) with \(|Q|\geq \delta |A|\).
\end{theorem}

This proportional quasi-independence is clearly a necessary condition for a set to be a finite union of quasi-independent sets, so Pisier's theorem provides support for such a decomposition existing.

Here we study abelian groups of bounded torsion, that is, groups for which there is an integer \(N\) with \(N\gamma=0\) for every \(\gamma\in\Gamma\). This case has a long history. Malliavin-Brameret and Malliavin \cite{MM} proved the prime exponent case in 1967, for groups of the form \((\Z/p\Z)^I\). In that setting \(\Gamma\) is a vector space over \(\Z/p\Z\), and the proof uses the Rado--Horn theorem, which characterizes when a finite set of vectors can be decomposed into independent classes \cite{R,H}; we use and discuss the Rado--Horn theorem further below. Varopoulos's 1968 lecture notes state, without giving a complete proof, that the Rado--Horn argument can be extended to groups of squarefree exponent \cite{V}. Bourgain later gave a proof \cite{B83} of the squarefree case using stochastic process methods; one of the results proved in that paper is a projection theorem that we will use here. This projection theorem reduces the general bounded torsion problem to groups of the form \((\Z/p^s\Z)^I\) for a fixed prime power \(p^s\). We explain this reduction in Section~\ref{sec:bourgain-reduction}.

It has long been known that the prime-power case cannot be resolved by the same approach as the prime case. The prime case crucially relies on the vector space structure of \((\Z/p\Z)^I\), which is a vector space over \(\Z/p\Z\). This structure is no longer available in the prime-power case, where \((\Z/p^s\Z)^I\) for \(s\geq 2\) is a module over \(\Z/p^s\Z\) rather than a vector space over a field, and direct analogs of the combinatorial arguments fail. The general bounded torsion problem was stated as an open problem in Varopoulos's 1968 lectures \cite{V}, Lindahl and Poulsen's 1971 book \cite{LP}, and L\'opez and Ross's 1975 book \cite{LR}. The special case \( (\Z/4\Z)^I\) was raised as an open problem in Bourgain's 1985 paper \cite{Bourgain1985ArithmeticDiameter} and his 2001 survey \cite{Bourgain2001LambdaP}. Our main result answers this question.

\begin{theorem}\label{thm:main}
Let \(G\) be a compact abelian group whose dual group \(\Gamma=\widehat G\) has bounded torsion. Then a set \(\Lambda\subseteq\Gamma\setminus\{0\}\) is Sidon if and only if \(\Lambda\) is a finite union of quasi-independent sets.
\end{theorem}

As discussed above, the implication that finite unions of quasi-independent sets are Sidon sets follows from the classical theory and Drury's theorem. Our contribution is the converse implication in prime-power groups, from which Theorem~\ref{thm:main} follows via Bourgain's reduction.

\begin{theorem}\label{thm:reduced}
Let \(G=(\Z/q\Z)^n\) for a prime power \(q=p^s\). Suppose \(A\subset G\) has the property that for every \(B\subseteq A\) there exists a quasi-independent subset \(Q\subseteq B\) with \(|Q|\geq \delta |B|\), for some \(\delta >0 \). Then \(A\) can be decomposed as
$$
        A=\bigsqcup_{i=1}^k A_i,
        \qquad
        k\leq \left\lceil \frac{s\log_2(p)}{\delta}\right\rceil,
$$
where each \(A_i\) is quasi-independent.
\end{theorem}

In applications to an infinite Sidon set, Pisier's arithmetic characterization supplies a constant \(\delta>0\) such that every finite subset satisfies the hypothesis of Theorem~\ref{thm:reduced}. Thus every finite subset of a Sidon set satisfies the proportional quasi-independence hypothesis.

The analogous problem with \(\Gamma=\Z\) is perhaps the most well-known open problem in the subject and is not addressed by our arguments.

\subsection{Overview of proof}
By Bourgain's projection theorem, it is enough to prove a uniform finite decomposition result in groups of prime-power exponent. Thus fix a prime power \(q=p^s\), let \(A\subset(\Z/q\Z)^n\) be finite, and write \(A=\{g_i:i\in[m]\}\). A subset of \(A\) is not quasi-independent exactly when it supports a nonzero signed relation
$$
        \sum_{i\in[m]}\epsilon_i g_i=0,
        \qquad
        \epsilon_i\in\{-1,0,1\}.
$$
We view such a relation as a vector \(\epsilon=(\epsilon_i)_{i\in[m]}\in(\Z/q\Z)^m\), with coordinates indexed by the elements of \(A\), and let \(\mathcal E\) be the collection of all signed zero-relations. Thus the goal is to partition \(A\) so that no vector in \(\mathcal E\) is supported inside a single class.

The difficulty is that \((\Z/q\Z)^m\) is not a vector space. We therefore reduce the vectors in \(\mathcal E\) modulo \(p\), obtaining vectors in \((\Z/p\Z)^m\). The main decomposition lemma, based on the Rado--Horn theorem, says roughly that it is enough to prove the following local dimension bound: for every \(T\subset[m]\), the reduced zero-relations supported inside \(T\) span a subspace of dimension at most a fixed proportion, bounded away from \(1\), of \(|T|\). Under this condition, one can partition \(A\) so that no reduced zero-relation is supported inside a single class; since nonzero signed coefficients \(\pm1\) remain nonzero modulo \(p\), the same partition also rules out signed zero-relations in \((\Z/q\Z)^m\).

It remains to prove this local dimension bound. The proportional quasi-independence hypothesis gives subgroup growth: if \(T\subset[m]\), then \(\{g_i:i\in T\}\) contains a large quasi-independent subset, whose subset sums are distinct, and hence \(\langle g_i:i\in T\rangle\) is large. This controls the number of coefficient relations supported on \(T\). Indeed, define
$$
        \Phi_T:(\Z/q\Z)^T\to \langle g_i:i\in T\rangle,
        \qquad
        (c_i)_{i\in T}\mapsto \sum_{i\in T} c_i g_i.
$$
Then \(\Phi_T\) is onto, and
$$
        |\ker \Phi_T|\,|\langle g_i:i\in T\rangle|=q^{|T|}.
$$
Thus subgroup growth gives an upper bound for \(|\ker\Phi_T|\), hence for the number of coefficient relations supported on \(T\). After reducing modulo \(p\), this gives the required dimension bound and completes the prime-power argument.

\subsection*{Use of AI-assisted tools.}
The author used AI-assisted tools, including ChatGPT, Gemini, and Claude, during the research underlying this paper, primarily as a sounding board, to compute and analyze large collections of examples, and to investigate proposed intermediate claims. These tools were also used for LaTeX assistance, formatting, and copyediting during the preparation of the manuscript. The mathematical arguments presented here are the author's own, and the author takes full responsibility for the contents of the paper.

\section{Notation}
For \(k\in\N\), write \([k]=\{1,\ldots,k\}\). For an integer \(q\geq 2\), write \(\Z_q=\Z/q\Z\). If \(A\) is a subset of an abelian group, \(\langle A\rangle\) denotes the subgroup generated by \(A\). If \(q=p^s\) and \(x\in\Z_q^T\), write \(\underline{x}\in\F_p^T\) for its reduction modulo \(p\). We write \(\operatorname{supp}(x)\) for the set of coordinates on which \(x\) is nonzero, and \(\Span\) always denotes linear span over the field in question.

\section{Linear relations and growth}
Let \(q=p^s\) be a prime power and let \(G=\Z_q^n\). A subset \(A\subset G\) is quasi-independent if the only choice of coefficients \(\epsilon_a\in\{-1,0,1\}\) satisfying \(\sum_{a\in A}\epsilon_a a=0\) is the trivial choice \(\epsilon_a=0\) for every \(a\in A\).

In this section we reduce Theorem~\ref{thm:reduced} to a linear algebra statement, which will be proved in the following sections. For \(A\subset G\), recall that \(\langle A\rangle=\{\sum_{a\in A}c_a a:c_a\in\Z_q\}\). Clearly, if \(A\) is quasi-independent, then its \(2^{|A|}\) subset sums are all distinct, and hence \(|\langle A\rangle|\geq 2^{|A|}\). More generally, if \(A\) contains a quasi-independent subset of size at least \(\delta |A|\), then \(|\langle A\rangle|\geq 2^{\delta |A|}\).

\begin{lemma}\label{lem:kernel-size}
Let \(q=p^s\) be a prime power and let \(G=\Z_q^n\). Let \(A\subset G\) have the property that for every \(B\subseteq A\), there exists a quasi-independent subset \(Q\subseteq B\) with \(|Q|\geq \delta |B|\). Then the number of coefficient vectors \(c=(c_a)_{a\in A}\in\Z_q^A\) satisfying \(\sum_{a\in A}c_a a=0\) is at most \(q^{|A|}/2^{\delta |A|}\).
\end{lemma}

\begin{proof}
Consider the homomorphism \(\Phi:\Z_q^A\to \langle A\rangle\) given by \(\Phi((c_a)_{a\in A})=\sum_{a\in A}c_a a\). This map is onto by the definition of \(\langle A\rangle\). Therefore every element of \(\langle A\rangle\) has the same number of preimages, namely \(|\ker\Phi|\), and so
$$
        q^{|A|}=|\Z_q^A|=|\ker\Phi|\,|\langle A\rangle|.
$$
By the proportional quasi-independence hypothesis, applied to \(B=A\), the set \(A\) contains a quasi-independent subset of size at least \(\delta |A|\). Hence \(|\langle A\rangle|\geq 2^{\delta |A|}\). It follows that \(|\ker\Phi|=q^{|A|}/|\langle A\rangle|\leq q^{|A|}/2^{\delta |A|}\). Since \(\ker\Phi\) is exactly the set of coefficient vectors \(c\in\Z_q^A\) satisfying \(\sum_{a\in A}c_a a=0\), this proves the lemma.
\end{proof}

We now need the following lemma, which we will use to relate relations mod~\(p\) and mod~\(p^s\). For \(x\in \Z_q^T\), write \(\underline{x}\in \F_p^T\) for its reduction modulo \(p\).

\begin{lemma}\label{lem:redModP}
Let \(q=p^s\), and let \(u_1,\ldots,u_d\in \Z_q^T\). Suppose that their reductions \(\underline{u_1},\ldots,\underline{u_d}\in \F_p^T\) are linearly independent over \(\F_p\). Then \(u_1,\ldots,u_d\) generate a subgroup of \(\Z_q^T\) of size \(q^d=p^{sd}\).
\end{lemma}

\begin{proof}
This is equivalent to showing that the map \(\Z_q^d\to\Z_q^T\) given by \((a_1,\ldots,a_d)\mapsto \sum_{j=1}^d a_j u_j\) is injective. Suppose \(\sum_{j=1}^d a_j u_j=0\) in \(\Z_q^T\). Reducing modulo \(p\), we get \(\sum_{j=1}^d \underline{a_j}\,\underline{u_j}=0\) in \(\F_p^T\). Since \(\underline{u_1},\ldots,\underline{u_d}\) are independent over \(\F_p\), each \(\underline{a_j}=0\). Thus each \(a_j\) is divisible by \(p\), say \(a_j=pb_j\).

Then \(p\sum_{j=1}^d b_j u_j=0\) in \((\Z/p^s\Z)^T\). Equivalently, the image of \(\sum_{j=1}^d b_j u_j\) in \((\Z/p^{s-1}\Z)^T\) is zero. Reducing this congruence modulo \(p\), the same independence argument shows that each \(b_j\) is divisible by \(p\). Hence each \(a_j\) is divisible by \(p^2\). Repeating this argument \(s\) times shows that each \(a_j\) is divisible by \(p^s\), so \(a_j=0\) in \(\Z_q\). Thus the map is injective, and its image has size \(|\Z_q^d|=p^{sd}\).
\end{proof}

For \(T\subseteq A\), let \(\mathcal E(T)\) denote the set of nonzero signed zero-relations supported on \(T\), viewed as vectors in \(\Z_q^T\). Thus \(\epsilon\in\mathcal E(T)\) means \(\epsilon=(\epsilon_a)_{a\in T}\), each \(\epsilon_a\in\{-1,0,1\}\), not all \(\epsilon_a\) are zero, and \(\sum_{a\in T}\epsilon_a a=0\). Let \(\underline{\mathcal E(T)}\subset\F_p^T\) denote the set of reductions of these vectors modulo \(p\).

\begin{corollary}\label{cor:local-dim-bound}
Let \(q=p^s\), and let \(A\subset \Z_q^n\) satisfy the proportional quasi-independence hypothesis of Theorem~\ref{thm:reduced}. For every \(T\subseteq A\), set \(d(T):=\dim_{\F_p}\Span \underline{\mathcal E(T)}\). Then
$$
        d(T)\leq \left(1-\frac{\delta}{s\log_2 p}\right)|T|.
$$
In particular, if \(M=\left\lceil s\log_2(p)/\delta\right\rceil\), then \(d(T)\leq \frac{M-1}{M}|T|\).
\end{corollary}

\begin{proof}
Fix \(T\subseteq A\), and write \(d=d(T)\). Choose \(\epsilon_1,\ldots,\epsilon_d\in\mathcal E(T)\) such that \(\underline{\epsilon_1},\ldots,\underline{\epsilon_d}\) are linearly independent over \(\F_p\). By Lemma~\ref{lem:redModP}, these \(\epsilon_j\)'s generate a subgroup of \(\Z_q^T\) of size \(q^d\).

Each \(\epsilon_j\) is a zero-relation among the elements of \(T\), so every element of this subgroup is also a coefficient relation among the elements of \(T\). Therefore this subgroup lies inside the kernel of the coefficient map \(\Z_q^T\to \langle T\rangle\). By Lemma~\ref{lem:kernel-size}, applied to \(T\), this kernel has size at most \(q^{|T|}/2^{\delta |T|}\). Hence
$$
        q^d\leq \frac{q^{|T|}}{2^{\delta |T|}}.
$$
Since \(q=p^s\), taking logarithms gives \(sd\log_2 p\leq s|T|\log_2 p-\delta |T|\), and therefore
$$
        d\leq \left(1-\frac{\delta}{s\log_2 p}\right)|T|.
$$
Finally, if \(M=\lceil s\log_2(p)/\delta\rceil\), then \(1/M\leq \delta/(s\log_2 p)\), so \(d(T)\leq (M-1)|T|/M\).
\end{proof}

\section{The Rado--Horn theorem and a support-partition consequence}

We begin by recalling a classical theorem proved by Rado~\cite{R} in 1942
and independently rediscovered by Horn~\cite{H} in 1955.
This result is the key tool used by Malliavin-Brameret and Malliavin~\cite{MM}
in their proof of the prime exponent case; we will use it in a less direct way.

\begin{theorem}[Rado--Horn\footnote{The original statements of Rado~\cite{R} and Horn~\cite{H} are for sets of vectors rather than indexed lists (which allow vectors to be repeated). This slightly more general form can be found, for instance, in the work of Casazza and Peterson~\cite{CasazzaPeterson2012}.}]
Let $I$ be a finite index set, let $(v_i)_{i \in I}$ be an indexed list of vectors in a vector
space $V$ over a field $\F$, and let $k \geq 1$. Suppose that
$$
|\{i \in I : v_i \in L\}| \leq k \dim L
$$
for every subspace $L \leq V$. Then $I$ admits a partition
$I = I_1 \sqcup \cdots \sqcup I_k$ such that $\{v_i : i \in I_j\}$ is
linearly independent for each $j$.
\end{theorem}

The necessity of the condition is immediate: if such a partition exists, then
for any subspace $L$ each class $I_j$ contributes at most $\dim L$ indices
to $\{i \in I : v_i \in L\}$, and summing over the $k$ classes gives
$|\{i \in I : v_i \in L\}| \leq k \dim L$.
The theorem asserts that this necessary condition is also sufficient.

We will use Rado--Horn via the following consequence, which is something of a dual formulation. The key feature is that the spanning condition in
the conclusion is imposed on a specified collection of vectors, not on the
entire subspace they generate; this distinction is what makes the lemma
applicable in our setting.

\begin{lemma}\label{lem:complement-radohorn}
Let $X$ be finite, let $(v_x)_{x\in X}$ be a finite indexed list of vectors
spanning a finite-dimensional vector space $W$ over $\F$, and let $k\geq 2$.
Suppose that for every proper subspace $L<W$,
$$
        |\{x\in X:v_x\notin L\}|
        \geq
        \frac{k}{k-1}\bigl(\dim W-\dim L\bigr).
$$
Then $X$ admits a partition $X=X_1\sqcup\cdots\sqcup X_k$ such that, for each
$j$,
$$
        \Span\{v_x:x\notin X_j\}=W.
$$
\end{lemma}

\begin{proof}
Put $d = \dim W$. Define the linear map $T : \F^X \to W$ by
$T(\mathbf{f}) = \sum_{x \in X} \mathbf{f}_x v_x$, where $\mathbf{f}_x$
denotes the $x$-th coordinate of $\mathbf{f} \in \F^X$. Since the $v_x$ span
$W$, the map $T$ is surjective, so the rank-nullity theorem gives $\dim\ker T = |X| - d$.

Fix a basis $\mathbf{g}_1, \ldots, \mathbf{g}_{|X|-d}$ of $\ker T \leq \F^X$,
where $(\mathbf{g}_i)_x \in \F$ denotes the $x$-th coordinate of $\mathbf{g}_i$,
so that $\sum_{x \in X}(\mathbf{g}_i)_x v_x = 0$ for each $i$.
Form the $(|X|-d) \times |X|$ matrix $K$ whose rows are $\mathbf{g}_1, \ldots,
\mathbf{g}_{|X|-d}$:
$$
K \;=\;
\bordermatrix{
          & x_1                        & x_2                        & \cdots & x_{|X|}                        \cr
\mathbf{g}_1     & (\mathbf{g}_1)_{x_1}        & (\mathbf{g}_1)_{x_2}        & \cdots & (\mathbf{g}_1)_{x_{|X|}}        \cr
\mathbf{g}_2     & (\mathbf{g}_2)_{x_1}        & (\mathbf{g}_2)_{x_2}        & \cdots & (\mathbf{g}_2)_{x_{|X|}}        \cr
\vdots           & \vdots                      & \vdots                      & \ddots & \vdots                          \cr
\mathbf{g}_{|X|-d} & (\mathbf{g}_{|X|-d})_{x_1} & (\mathbf{g}_{|X|-d})_{x_2} & \cdots & (\mathbf{g}_{|X|-d})_{x_{|X|}} \cr
}
$$
for any enumeration $x_1,\ldots,x_{|X|}$ of $X$. Column $x$ of $K$ is the vector
$$u_x =
((\mathbf{g}_1)_x, (\mathbf{g}_2)_x, \ldots, (\mathbf{g}_{|X|-d})_x)^\top
\in \F^{|X|-d}.$$
We apply the Rado--Horn theorem to the indexed list
$(u_x)_{x \in X} \subset \F^{|X|-d}$.

For $C \subseteq X$, write $D = X\setminus C$ and let $K_C$ denote the
$(|X|-d)\times|C|$ submatrix of $K$ with columns indexed by $C$. We claim
\begin{equation}\label{eq:rankid}
        \dim\Span\{u_x : x \in C\} \;=\; |C| - d + \dim\Span\{v_x : x \in D\}.
\end{equation}
We prove this by computing the rank of $K_C$ two ways. The column rank of
$K_C$ equals $\dim\Span\{u_x : x \in C\}$ since the columns of $K_C$ are
exactly the $u_x$ for $x \in C$. For the row rank, note that row $i$ of $K_C$
is $\mathbf{g}_i$ restricted to its $C$-coordinates. A linear combination
$\sum_i \lambda_i \mathbf{g}_i$ restricts to zero on $C$ if and only if the
vector $\mathbf{g} = \sum_i \lambda_i \mathbf{g}_i$ has $(\mathbf{g})_x = 0$
for all $x \in C$, i.e., all nonzero coordinates of $\mathbf{g}$ lie in $D$.
Since each $\mathbf{g}_i \in \ker T$ and $\ker T$ is a subspace, $\mathbf{g}
\in \ker T$, meaning $T(\mathbf{g}) = 0$. Since $\mathbf{g}$ is supported on
$D$, only the $D$-coordinates contribute to $T(\mathbf{g})$, so
$\sum_{x\in D}(\mathbf{g})_x v_x = 0$.

Conversely, any $\mathbf{h} \in \F^D$ with $\sum_{x\in D}\mathbf{h}_x v_x = 0$
extends to an element of $\ker T$ supported on $D$, hence with all
$C$-coordinates zero, giving a combination of rows that vanishes on $C$.
Therefore the combinations
$\sum_i\lambda_i\mathbf{g}_i$ vanishing on $C$ correspond exactly to
$\ker(T_D)$, where $T_D : \F^D \to W$ is defined by $T_D(\mathbf{h}) =
\sum_{x\in D}\mathbf{h}_x v_x$. Applying the rank-nullity theorem to $T_D : \F^D \to W$,
$$
        \dim\ker(T_D) = |D| - \dim\Span\{v_x : x\in D\}.
$$
Since $\ker T$ has dimension $|X|-d$, the rank-nullity theorem gives
$$
        \mathrm{rowrank}(K_C)
        = (|X|-d) - \dim\ker(T_D)
        = (|X|-d) - |D| + \dim\Span\{v_x:x\in D\}.
$$
Since row rank equals column rank and $|D| = |X|-|C|$, this gives~\eqref{eq:rankid}.

We now verify the Rado--Horn hypothesis, i.e., that $|C| \leq
k\,\dim\Span\{u_x : x \in C\}$ for every $C \subseteq X$. Fix $C$ and let $L
= \Span\{v_x : x \in D\} \leq W$. If $L = W$, then \eqref{eq:rankid} gives
$\dim\Span\{u_x : x \in C\} = |C|$ and the inequality is trivial. If $L < W$,
every $x$ with $v_x \notin L$ lies in $C$, so the hypothesis of the lemma
gives $|C| \geq \frac{k}{k-1}(d - \dim L)$. Since \eqref{eq:rankid} gives
$\dim\Span\{u_x : x \in C\} = |C| - d + \dim L$, this rearranges to
$$
        |C| \leq k(|C| - d + \dim L) = k\,\dim\Span\{u_x : x \in C\}.
$$
The Rado--Horn theorem therefore partitions $X$ into classes $X_1, \ldots, X_k$ with each
indexed list $(u_x)_{x \in X_j}$ linearly independent. Applying \eqref{eq:rankid}
with $C = X_j$ and using $\dim\Span\{u_x : x \in X_j\} = |X_j|$ gives
$\dim\Span\{v_x : x \notin X_j\} = d = \dim W$, hence
$\Span\{v_x : x \notin X_j\} = W$.
\end{proof}

\begin{theorem}\label{thm:supportpartition}
Let $X$ be finite, let $\mathcal C\subseteq \F^X\setminus\{0\}$, and, for
$Y\subseteq X$, define
$$
        \lambda_{\mathcal C}(Y)
        =
        \dim\Span\{c\in\mathcal C:\operatorname{supp} c\subseteq Y\}.
$$
Let $k\geq 2$. If $\lambda_{\mathcal C}(Y) \leq \frac{k-1}{k}|Y|$ for every
$Y\subseteq X$, then $X$ admits a partition $X=X_1\sqcup\cdots\sqcup X_k$
such that no vector $c\in\mathcal C$ has support contained in a single
class $X_j$.
\end{theorem}

\begin{proof}
We argue by induction on $|X|$. If $\mathcal C$ is empty the condition is
vacuously satisfied and any partition works.

The proof uses Lemma~\ref{lem:complement-radohorn}, which partitions an indexed list
of vectors in some ambient space $W$ into $k$ classes such that each class's
complement spans $W$. We want to apply this to produce a partition of $X$ such
that every $\mathbf{c} \in \mathcal C$ has a support point outside each
class. We assign to each $x \in X$ a vector $\mathbf{m}_x \in \F^r$
such that $x \in \operatorname{supp} \mathbf{c}$ if and only if
$\boldsymbol{\beta}_{\mathbf{c}} \cdot \mathbf{m}_x \neq 0$, where
$\boldsymbol{\beta}_{\mathbf{c}} \in \F^r$ is the coefficient vector of
$\mathbf{c}$. Given such vectors, if the complement of each class spans $\F^r$,
then for any nonzero $\boldsymbol{\beta}_{\mathbf{c}}$ there is a point outside
each class not orthogonal to $\boldsymbol{\beta}_{\mathbf{c}}$, which is
precisely a support point of $\mathbf{c}$ outside that class.

Let $\mathbf{b}_1, \ldots, \mathbf{b}_r$ be a basis of $R = \Span\mathcal C
\leq \F^X$, and form the $r \times |X|$ matrix whose rows are
$\mathbf{b}_1, \ldots, \mathbf{b}_r$, with $(\mathbf{b}_i)_x \in \F$ denoting
the $x$-th coordinate of $\mathbf{b}_i$. Its column $x$ is the vector
$\mathbf{m}_x = ((\mathbf{b}_1)_x, \ldots, (\mathbf{b}_r)_x)^\top \in \F^r$. Every
$\mathbf{c} \in \mathcal C$ lies in $R$, so $\mathbf{c} = \sum_{i=1}^r
\beta_i \mathbf{b}_i$ for unique coefficients, and the coefficient vector
$\boldsymbol{\beta}_{\mathbf{c}} = (\beta_1, \ldots, \beta_r)^\top \in \F^r$
satisfies
$$
        \mathbf{c}_x = \boldsymbol{\beta}_{\mathbf{c}} \cdot \mathbf{m}_x
        \quad \text{for every } x \in X.
$$
Thus $x \in \operatorname{supp} \mathbf{c}$ if and only if
$\boldsymbol{\beta}_{\mathbf{c}} \cdot \mathbf{m}_x \neq 0$, as required.
Any $x$ with $\mathbf{m}_x = 0$ satisfies $\mathbf{c}_x = 0$ for every
$\mathbf{c} \in \mathcal C$ and may be placed in any class; we therefore
restrict attention to $X^+ = \{x \in X : \mathbf{m}_x \neq 0\}$. The
hypothesis applied to $Y = X^+$ gives $|X^+| \geq \frac{k}{k-1}r$, and the
columns $(\mathbf{m}_x)_{x \in X^+}$ span $\F^r$, since otherwise some
nontrivial linear combination of the $\mathbf{b}_i$ would vanish on all of $X$.

We would like to apply Lemma~\ref{lem:complement-radohorn} directly to
$(\mathbf{m}_x)_{x \in X^+}$ in $\F^r$, but the lemma's hypothesis may fail
if some proper subspace $L \subsetneq \F^r$ contains too many
columns $\mathbf{m}_x$. We start by looking for the smallest subspace $W \leq \F^r$ such that the
columns landing in $W$ satisfy the lemma's hypothesis within $W$. We will apply
Lemma~\ref{lem:complement-radohorn} on $W$. Vectors that vanish on $W$ are passed to the
induction step. For each subspace $U \leq \F^r$, we write $X_U = \{x \in X^+ :
\mathbf{m}_x \in U\}$.
Choose $W \leq \F^r$ of minimal positive dimension satisfying $|X_W| \geq
\frac{k}{k-1}\dim W$. Since $\F^r$ satisfies this by the bound on
$|X^+|$, such a $W$ must exist.

We verify the hypothesis of Lemma~\ref{lem:complement-radohorn} for the indexed list
$(\mathbf{m}_x)_{x \in X_W}$ in the space $W$. Since every $x \in X_W
\subseteq X^+$ has $\mathbf{m}_x \neq 0$, no column lies in the zero subspace,
so all $|X_W| \geq \frac{k}{k-1}\dim W$ columns lie outside $\{0\}$, completing
the case $L = 0$. For any nonzero proper subspace $L \subsetneq W$, minimality
of $W$ gives $|X_L| < \frac{k}{k-1}\dim L$, and since $X_L \subseteq X_W$,
$$
        |\{x \in X_W : \mathbf{m}_x \notin L\}|
        = |X_W| - |X_L|
        > \frac{k}{k-1}(\dim W - \dim L).
$$
The vectors $(\mathbf{m}_x)_{x \in X_W}$ also span $W$, since if they spanned
only a proper subspace $W' \subsetneq W$ then $X_{W'} = X_W$ and $W'$ would
satisfy the defining inequality with smaller dimension, contradicting
minimality. Lemma~\ref{lem:complement-radohorn} therefore yields a partition
$X_W = P_1 \sqcup \cdots \sqcup P_k$ such that
\begin{equation}\label{eq:spans}
        \Span\{\mathbf{m}_x : x \in X_W \setminus P_j\} = W
        \quad \text{for every } j.
\end{equation}

We now show that~\eqref{eq:spans} gives the support condition for every
$\mathbf{c} \in \mathcal C$ that is not identically zero on $X_W$. Fix such
a $\mathbf{c}$, with coefficient vector $\boldsymbol{\beta}_{\mathbf{c}}$.
Since $\mathbf{c}$ is not identically zero on $X_W$, there exists $x_0 \in
X_W$ with $\boldsymbol{\beta}_{\mathbf{c}} \cdot \mathbf{m}_{x_0} \neq 0$,
so $\boldsymbol{\beta}_{\mathbf{c}}$ does not vanish on all of $W$. The set
$H = \{v \in W : \boldsymbol{\beta}_{\mathbf{c}} \cdot v = 0\}$ is therefore a
proper subspace of $W$. For each $j$,~\eqref{eq:spans} gives us that
$\{\mathbf{m}_x : x \in X_W \setminus P_j\}$ spans $W$, so these columns are
not all contained in $H$, and hence some $x_j \in X_W \setminus P_j$ satisfies
$\boldsymbol{\beta}_{\mathbf{c}} \cdot \mathbf{m}_{x_j} \neq 0$, meaning
$x_j \in \operatorname{supp} \mathbf{c} \setminus P_j$.

It remains to consider the vectors in $\mathcal C$ that vanish identically on
$X_W$. Their supports are contained in $Y = X^+ \setminus X_W$, and their
nonzero restrictions to $Y$ form a collection $\mathcal C_Y$. Fix $J \subseteq Y$.
An element of $\mathcal C_Y$ supported in $J$ is the restriction to $Y$ of
some $\mathbf{c}\in\mathcal C$ that vanishes on $X_W$. Since every vector in
$\mathcal C$ also vanishes on $X\setminus X^+$, and since the restriction of
$\mathbf{c}$ to $Y$ is supported in $J$, the original vector $\mathbf{c}$ is
supported in $J$ as a vector on $X$. Therefore
$\lambda_{\mathcal C_Y}(J) \leq \lambda_{\mathcal C}(J) \leq
\frac{k-1}{k}|J|$. Since $W \neq 0$ the set $X_W$ is nonempty, so $|Y| < |X|$,
and the induction hypothesis gives a partition $Y = Q_1 \sqcup \cdots \sqcup Q_k$
such that no vector in $\mathcal C_Y$ has support contained in a single $Q_j$.

Setting $Z_j = P_j \cup Q_j$ gives a partition of $X^+$. For any $\mathbf{c}
\in \mathcal C$ not identically zero on $X_W$, the preceding argument gives a
support point in $X_W \setminus P_j \subseteq Z_j^c$ for each $j$, so
$\operatorname{supp} \mathbf{c} \not\subseteq Z_j$. For any $\mathbf{c} \in
\mathcal C$ vanishing on $X_W$, its restriction lies in $\mathcal C_Y$, and
since $Z_j \cap Y = Q_j$ the inductive step gives $\operatorname{supp}
\mathbf{c} \not\subseteq Z_j$ for every $j$. Assigning the points of $X
\setminus X^+$ to any class completes the induction.
\end{proof}

\section{Proof of Theorem~\ref{thm:reduced}}

We now put the preceding results together. Let \(q=p^s\), let \(G=\Z_q^n\), and let \(A\subset G\) satisfy the hypothesis of Theorem~\ref{thm:reduced}; that is, every subset \(B\subseteq A\) contains a quasi-independent subset of size at least \(\delta |B|\).

\begin{proof}[Proof of Theorem~\ref{thm:reduced}]
Set
$$
        k=\left\lceil \frac{s\log_2(p)}{\delta}\right\rceil.
$$
If \(k=1\), then the hypothesis applied to \(B=A\) implies that \(A\) itself is quasi-independent, and there is nothing to prove. We may therefore assume \(k\geq 2\).

Let
$$
        \mathcal C=\underline{\mathcal E(A)}\subseteq \F_p^A
$$
be the collection of reductions modulo $p$ of the nonzero signed zero-relations in $A$. For $T\subseteq A$, we identify $\F_p^T$ with the coordinate subspace of $\F_p^A$ supported on $T$. With this convention, the vectors in $\mathcal C$ whose supports are contained in $T$ are exactly the vectors in $\underline{\mathcal E(T)}$: indeed, each nonzero signed coefficient is $\pm1$, hence remains nonzero after reduction modulo $p$. Therefore Corollary~\ref{cor:local-dim-bound} gives
$$
        \dim_{\F_p}\Span\{c\in\mathcal C:\operatorname{supp}c\subseteq T\}
        =
        \dim_{\F_p}\Span \underline{\mathcal E(T)}
        \leq \frac{k-1}{k}|T|
$$
for every $T\subseteq A$.

This is exactly the hypothesis of Theorem~\ref{thm:supportpartition}, applied over the field $\F_p$ with $X=A$ and with the forbidden collection $\mathcal C$. Hence Theorem~\ref{thm:supportpartition} gives a partition
$A=\bigsqcup_{j=1}^k A_j$ such that no vector in $\mathcal C$ has support contained entirely in one class $A_j$.

It remains only to translate this back from reduced relations to the original group \(G=\Z_q^n\). Suppose, for contradiction, that some class \(A_j\) is not quasi-independent. Then there is a nonzero choice of coefficients \(\epsilon_a\in\{-1,0,1\}\), supported on \(A_j\), such that
$$
        \sum_{a\in A_j}\epsilon_a a=0
        \qquad \text{in } \Z_q^n.
$$
Equivalently, this coefficient vector is an element of \(\mathcal E(A_j)\). Since each nonzero coefficient is \(\pm1\), its reduction modulo \(p\) is still nonzero. Therefore \(\underline{\epsilon}\) is a nonzero element of \(\underline{\mathcal E(A_j)}\) with support contained in \(A_j\), contradicting the partition property.

Thus every class \(A_j\) is quasi-independent, and we have decomposed
\(A\) as \(A=\bigsqcup_{j=1}^k A_j\) with \(k\leq \left\lceil s\log_2(p)/\delta\right\rceil\), as claimed.
\end{proof}

\section{Reduction from bounded torsion to prime powers}
\label{sec:bourgain-reduction}

We now explain how Theorem~\ref{thm:reduced}, which was proved for groups of the form $(\Z/p^s\Z)^n$, implies the general bounded torsion result.

We use the following projection theorem of Bourgain.

\begin{theorem}[Bourgain, Proposition 2.3 of \cite{B83}]
\label{thm:bourgain-reduction}
Let $G_1$ and $G_2$ be compact abelian groups, and identify
$$
(G_1\times G_2)^\wedge
=
\widehat{G_1}\times \widehat{G_2}.
$$
Let $\pi_i\colon \widehat{G_1}\times \widehat{G_2}\to \widehat{G_i}$ be the two coordinate projections. If $\Lambda\subseteq \widehat{G_1}\times \widehat{G_2}$ is Sidon, then $\Lambda$ is a finite union of sets $\Lambda_\alpha$ such that, for each $\alpha$, one of the two coordinate projections is one-to-one on $\Lambda_\alpha$, and the projected set $\pi_i(\Lambda_\alpha)$ is Sidon in $\widehat{G_i}$.
\end{theorem}

Bourgain's proof uses Rider's characterization of Sidon sets via random signs and a Gaussian comparison argument. A similar argument in a related context appears in \cite{BL}. This theorem lets us reduce questions about Sidon sets in a product group to questions about Sidon sets in one coordinate group.

We will use the fact that quasi-independence pulls back through an injective coordinate projection. Indeed, let $E\subseteq \widehat{G_1}\times\widehat{G_2}$, and suppose that one coordinate projection, say $\pi_i$, is one-to-one on $E$. If $\pi_i(E)$ is quasi-independent, then $E$ is quasi-independent. If
$$
\sum_{\lambda\in E}\epsilon_\lambda \lambda=0,
\qquad \epsilon_\lambda\in{-1,0,1},
$$
then applying $\pi_i$ gives
$$
\sum_{\lambda\in E}\epsilon_\lambda \pi_i(\lambda)=0.
$$
Since $\pi_i$ is one-to-one on $E$, the elements $\pi_i(\lambda)$ are distinct. Since $\pi_i(E)$ is quasi-independent, all coefficients $\epsilon_\lambda$ are zero. Thus $E$ is quasi-independent. Consequently, if $\pi_i$ is one-to-one on $\Lambda_\alpha$, then any partition of $\pi_i(\Lambda_\alpha)$ into quasi-independent sets pulls back to a partition of $\Lambda_\alpha$ into quasi-independent sets.

Now let $\Gamma$ be a discrete abelian group of bounded torsion. Choose $N$ such that $N\gamma=0$ for every $\gamma\in\Gamma$, and write
$$
N=\prod_{j=1}^r p_j^{s_j}.
$$
For each $j$, let
$$
\Gamma_j={\gamma\in\Gamma:p_j^{s_j}\gamma=0}.
$$
The standard decomposition into the subgroups annihilated by the prime-power factors of $N$ gives
$$
\Gamma=\Gamma_1\times\cdots\times\Gamma_r.
$$
Thus every element of $\Gamma$ has a unique decomposition into components lying in the subgroups $\Gamma_j$.

Applying Bourgain's theorem iteratively in the dual notation, first to
$$
\Gamma_1\times(\Gamma_2\times\cdots\times\Gamma_r),
$$
then to the remaining product factor, and so on, we obtain a finite decomposition of $\Lambda$ into sets $\Lambda_\alpha$ with the following property: for each $\alpha$, there is an index $j$ such that the coordinate projection $\pi_j$ is one-to-one on $\Lambda_\alpha$, and $\pi_j(\Lambda_\alpha)$ is a Sidon subset of $\Gamma_j$.

By the observation above, once $\pi_j(\Lambda_\alpha)$ has been partitioned into quasi-independent sets, the same partition pulls back to a partition of $\Lambda_\alpha$ into quasi-independent sets. Therefore it is enough to prove the conclusion for Sidon subsets of groups of exponent dividing a fixed prime power.

It remains to connect this prime-power reduction with Theorem~\ref{thm:reduced}, which was stated for finite subsets of $(\Z/p^s\Z)^n$. Let $B$ be a finite subset of a group in which every element has order dividing $p^s$. Then $H=\langle B\rangle$ is a finite abelian group of exponent dividing $p^s$. Hence
$$
H\cong \prod_{\ell=1}^n \Z/p^{r_\ell}\Z
$$
with $1\leq r_\ell\leq s$ for every $\ell$. Each cyclic group $\Z/p^{r_\ell}\Z$ embeds into $\Z/p^s\Z$ by
$$
x \mapsto p^{s-r_\ell}x.
$$
Taking the product of these embeddings gives an injective homomorphism
$$
H\hookrightarrow (\Z/p^s\Z)^n.
$$
Because this map is an injective homomorphism, it preserves signed relations in both directions. Therefore a subset of $B$ is quasi-independent if and only if its image in $(\Z/p^s\Z)^n$ is quasi-independent.

\begin{proof}[Proof of Theorem~\ref{thm:main} from Theorem~\ref{thm:reduced}]
If $\Lambda$ is a finite union of quasi-independent sets, then $\Lambda$ is Sidon by the classical Riesz product argument and Drury's theorem \cite{Drury}.

Conversely, suppose $\Lambda$ is Sidon. Bourgain's theorem reduces $\Lambda$ to finitely many sets whose injective projections are Sidon subsets of groups of exponent dividing a fixed prime power. Since a finite union of finite decompositions is still finite, it is enough to prove the result for a Sidon set $\Lambda$ contained in a group $\Gamma$ of exponent dividing $p^s$, where $p^s$ is fixed.

By Pisier's arithmetic characterization, there is a constant $\delta>0$ such that every finite set $B\subseteq\Lambda$ contains a quasi-independent subset of size at least $\delta |B|$. Fix such a finite set $B$. The subgroup $\langle B\rangle$ embeds into $(\Z/p^s\Z)^n$ for some $n$, and this embedding preserves quasi-independence. Hence the image of $B$ satisfies the hypothesis of Theorem~\ref{thm:reduced} with the same value of $\delta$. Therefore the image of $B$, and hence $B$ itself, admits a partition into at most
$$
M=\left\lceil \frac{s\log_2 p}{\delta}\right\rceil
$$
quasi-independent parts. The important point is that $M$ depends only on $p^s$ and $\delta$, not on the finite set $B$.

It remains to pass from finite subsets of $\Lambda$ to $\Lambda$ itself. Let $[M]={1,\ldots,M}$, and consider all assignments
$$
\varphi:\Lambda\to [M].
$$
Such an assignment determines a partition of $\Lambda$ into the parts $\varphi^{-1}(1),\ldots,\varphi^{-1}(M)$.

For a finite set $B\subseteq\Lambda$, say that an assignment $\varphi$ is good on $B$ if each set
$$
B\cap \varphi^{-1}(j),
\qquad 1\leq j\leq M,
$$
is quasi-independent. We have just shown that for every finite $B\subseteq\Lambda$, there is at least one assignment good on $B$. More is true: if $B_1,\ldots,B_t$ are finitely many finite subsets of $\Lambda$, then there is one assignment that is good on all of them at once. Indeed, apply the finite result to $B_1\cup\cdots\cup B_t$, and then extend the resulting assignment arbitrarily to the rest of $\Lambda$.

Now use compactness of the product space $[M]^\Lambda$. Since $[M]$ is finite, $[M]^\Lambda$ is compact in the product topology. The preceding paragraph says that every finite list of finite requirements can be satisfied simultaneously. Compactness therefore gives a single assignment $\varphi:\Lambda\to[M]$ that is good on every finite subset $B\subseteq\Lambda$.

We claim that the parts $\varphi^{-1}(1),\ldots,\varphi^{-1}(M)$ are quasi-independent. If one part contained a nontrivial signed zero-relation, that relation would involve only finitely many elements of $\Lambda$. Taking $B$ to be this finite set of elements would contradict the fact that $\varphi$ is good on every finite subset. Hence every part is quasi-independent, and $\Lambda$ is a union of at most $M$ quasi-independent sets.
\end{proof}


\begin{thebibliography}{99}

\bibitem{B83}
J.~Bourgain,
\emph{Propri\'et\'es de d\'ecomposition pour les ensembles de Sidon},
Bull.\ Soc.\ Math.\ France \textbf{111} (1983), no.~4, 421--428.

\bibitem{Brp}
J.~Bourgain,
\emph{Sidon sets and Riesz products},
Ann.\ Inst.\ Fourier (Grenoble) \textbf{35} (1985), no.~1, 137--148.

\bibitem{Bourgain1985ArithmeticDiameter}
J.~Bourgain,
\emph{Subspaces of $L^\infty_N$, arithmetical diameter and Sidon sets},
in \emph{Probability in Banach Spaces V},
Lecture Notes in Math., vol.~1153, Springer, Berlin, 1985, 96--127.

\bibitem{Bourgain2001LambdaP}
J.~Bourgain,
\emph{$\Lambda_p$-sets in analysis: results, problems and related aspects},
in \emph{Handbook of the Geometry of Banach Spaces, Vol.~I},
North-Holland, Amsterdam, 2001, 195--232.

\bibitem{BL}
J.~Bourgain and M.~Lewko,
\emph{Sidonicity and variants of Kaczmarz's problem},
Ann.\ Inst.\ Fourier (Grenoble) \textbf{67} (2017), no.~3, 1321--1352.

\bibitem{CasazzaPeterson2012}
P.~G.~Casazza and J.~Peterson,
\emph{An elementary, illustrative proof of the Rado--Horn theorem},
Linear Algebra Appl.\ \textbf{437} (2012), no.~10, 2523--2537.

\bibitem{Drury}
S.~W.~Drury,
\emph{Unions of sets of interpolation},
in \emph{Conference on Harmonic Analysis},
Lecture Notes in Math., vol.~266,
Springer, Berlin, 1972, 23--33.

\bibitem{GH}
C.~C.~Graham and K.~E.~Hare,
\emph{Interpolation and Sidon Sets for Compact Groups},
CMS Books in Mathematics,
Springer, New York, 2013.

\bibitem{H}
A.~Horn,
\emph{A characterization of unions of linearly independent sets},
J.\ London Math.\ Soc.\ \textbf{30} (1955), 494--496.

\bibitem{LP}
L.-\AA.~Lindahl and F.~Poulsen, editors,
\emph{Thin sets in harmonic analysis},
Lecture Notes in Pure and Applied Mathematics, vol.~2,
Marcel Dekker, New York, 1971.

\bibitem{LR}
J.~M.~L\'opez and K.~A.~Ross,
\emph{Sidon Sets},
Lecture Notes in Pure and Applied Mathematics, Vol.~13,
Marcel Dekker, Inc., New York, 1975.

\bibitem{MM}
M.-P.~Malliavin-Brameret and P.~Malliavin,
\emph{Caract\'erisation arithm\'etique d'une classe d'ensembles de Helson},
C.~R.\ Acad.\ Sci.\ Paris S\'er.\ A--B \textbf{264} (1967), A192--A193.

\bibitem{Pisier1981Sidon}
G.~Pisier,
\emph{De nouvelles caract\'erisations des ensembles de Sidon},
in \emph{Mathematical Analysis and Applications, Part B},
Adv.\ in Math.\ Suppl.\ Stud., vol.~7B, Academic Press, New York, 1981, 685--726.

\bibitem{Pisier1983Arithmetic}
G.~Pisier,
\emph{Arithmetic characterizations of Sidon sets},
Bull.\ Amer.\ Math.\ Soc.\ (N.S.) \textbf{8} (1983), no.~1, 87--89.

\bibitem{Pisier1983Entropy}
G.~Pisier,
\emph{Conditions d'entropie et caract\'erisations arithm\'etiques des ensembles de Sidon},
in \emph{Topics in Modern Harmonic Analysis, Vol.~II} (Torino/Milano, 1982),
Ist.\ Naz.\ Alta Mat.\ Francesco Severi, Rome, 1983, 911--944.

\bibitem{Pisier2016Subgaussian}
G.~Pisier,
\emph{Subgaussian sequences in probability and Fourier analysis},
Graduate J.\ Math.\ \textbf{1} (2016), 59--78.

\bibitem{R}
R.~Rado,
\emph{A theorem on independence relations},
Q.\ J.\ Math.\ Oxford Ser.\ \textbf{13} (1942), 83--89.

\bibitem{V}
N.~Th.~Varopoulos,
\emph{Some combinatorial problems in harmonic analysis},
lecture notes by D.~Salinger from a course at the Summer School in Harmonic Analysis,
University of Warwick, July 1--13, 1968,
Mathematics Institute, University of Warwick.

\end{thebibliography}
\end{document}